\documentclass[a4paper,12pt]{article}

\usepackage{geometry}
\usepackage{latexsym}
\usepackage{amssymb}
\usepackage{amsthm}
\usepackage{amsmath,amsfonts}
\usepackage{xcolor}
\usepackage{float}
\usepackage{url}
\usepackage[hidelinks]{hyperref}
\usepackage{titlesec}
\titleformat{\section}{\normalfont\Large\bfseries\raggedright}{\thesection}{1em}{}

\def\titlerunning#1{\gdef\titrun{#1}}
\makeatletter
\def\author#1{\gdef\autrun{\def\and{\unskip, }#1}\gdef\@author{#1}}
\def\address#1{{\def\and{\\\hspace*{18pt}}\renewcommand{\thefootnote}{}%
\footnote {#1}}%
\markboth{\autrun}{\titrun}}
\makeatother
\def\email#1{\hspace*{4pt}{\em e-mail}: #1}

\newtheorem{thm}{Theorem}[section]
\newtheorem{prop}[thm]{Proposition}

\newtheorem{cor}[thm]{Corollary}
\theoremstyle{definition}
\newtheorem{rem}[thm]{Remark}

\titlerunning{}

\title{Quasi-strongly regular digraphs constructed from transitive groups of degree $n\leq 110$}
\author{Vedrana Mikuli\'c Crnkovi\'c and Matea Zubovi\'c \v Zutolija}

\begin{document}
\maketitle

\address{V. Mikuli\'c Crnkovi\'c, M. Zubovi\'c \v Zutolija: Faculty of Mathematics, University of Rijeka, Radmile Matej\v ci\'c 2, 51000 Rijeka, Croatia;
\email{\{vmikulic,matea.zubovic\}@math.uniri.hr}
}

\begin{abstract}
In this paper we present a method for constructing directed regular graphs from a transitive permutation group. This method is a generalization of a construction method for transitive 1-designs from a finite group, described in \cite{dean1}. Using this construction, we prove the existence of directed strongly regular graphs with parameters $(72,25,15,6,10)$, $(96,11,4,3,1)$, $(96,22,16,8,4)$, $(96,26,11,7,7)$, $(96,29,23,10,8)$, $(96,42,32,16,20)$, $(96,45,35,22,20)$ and $(165,60,36,23,21)$. Finally, we classify quasi-strongly regular digraphs arising from transitive permutation groups of degree at most $30$ and from primitive permutation groups of degrees from $31$ to $110$ and rank at most $30$.

\end{abstract}

\bigskip



\section{Introduction}
\label{intro}
Directed strongly regular graphs were introduced in \cite{duval} as a directed version of strongly regular graphs. For more recent results on directed strongly regular graphs we refer the reader to \cite{2024,2025,BP2024,zrinski,CSZ2026,novi_108,MR2026}. A classification of regular digraphs, normally regular digraphs and strongly regular digraphs is described in \cite{nrds}. As a generalization of directed strongly regular graphs, the authors in \cite{quasi} introduced quasi-strongly regular digraphs. Deza digraphs, which were introduced in \cite{deza} and further studied in \cite{deza3,deza1,deza2}, provide another generalization of directed strongly regular graphs. 

In this paper, we consider directed quasi-strongly regular graphs constructed from transitive permutation groups of degree up to 30 and primitive groups of degree $31\leq n\leq 110$. Using the \textit{GAP Transitive Groups Library} \cite{transgrp} and \textit{GAP Primitive Groups Library} \cite{primgrp}, we construct directed quasi-strongly regular graphs on $n$ vertices, $n\in\{3,\dots,110\}$. Starting from a transitive action of a group $G$ on a set $\Omega$ and a union $\Delta$ of suborbits of a stabilizer $G_{\alpha}$, we obtain a vertex-transitive regular digraph whose out-neighborhoods are precisely the sets $g.\Delta$, $g\in G$. A simple situation occurs when $\Delta$ is a single non-diagonal $G_{\alpha}$-orbit; in this situation, the resulting digraph is quasi-strongly regular.

There are many open parameter sets for directed strongly regular graphs on up to 110 vertices on the website of A. Brouwer and S. Hobart (see \cite{dsrgbaza}). Using our construction, we obtain directed strongly regular graphs with parameters $(72,25,15,6,10)$, $(96,11,4,3,1)$, $(96,22,16,8,4)$, $(96,26,11,7,7)$, $(96,29,23,10,8)$, $(96,42,32,16,20)$, $(96,45,35,22,20)$ and $(165,60,36,23,21)$. To the best of our knowledge, all of these directed strongly regular graphs are new. 

The paper is organized as follows: In Section 2, we give basic definitions and properties of combinatorial designs and directed graphs used in the construction presented in this paper. We also describe the construction method of transitive 1-designs from a finite group described in \cite{dean1}. In Section 3, we present the construction of vertex-transitive directed regular graphs, analyze its single-orbit case, and describe the behavior of complements of the quasi-strongly regular digraphs obtained in this way. In Section 4, we prove the existence of directed strongly regular graphs with parameters $(72,25,15,6,10)$, $(96,11,4,3,1)$, $(96,22,16,8,4)$, $(96,26,11,7,7)$, $(96,29,23,10,8)$, $(96,42,32,16,20)$, $(96,45,35,22,20)$ and $(165,60,36,23,21)$. We also give information about all constructed directed strongly regular graphs, up to isomorphism. In Section 5, we describe directed quasi-strongly regular graphs constructed under the action of transitive permutation groups of degree $n\in\{3,\dots,30\}$ and primitive permutation groups of degree $n\in\{31,\dots,110\}$, considering only primitive groups of rank at most $30$. We also give information about all constructed directed quasi-strongly regular graphs, up to isomorphism.

\section{Preliminaries}\label{sec:1}
\subsection{$t$-designs}\label{subsec:1.1}
An incidence structure is an ordered triple $\mathcal{D}=(\mathcal{P},\mathcal{B},\mathcal{I})$, where $\mathcal{P}$ and $\mathcal{B}$ are non-empty, disjoint sets and $\mathcal{I}\subseteq \mathcal{P}\times \mathcal{B}$. The elements of the set $\mathcal{P}$ are called points, the elements of the set $\mathcal{B}$ are called blocks and $\mathcal{I}$ is called the incidence relation. If $|\mathcal{P}|=|\mathcal{B}|$, the incidence structure is called symmetric. The incidence matrix of an incidence structure is a $b\times v$ matrix $[m_{ij}]$ where $b$ and $v$ are the number of blocks and points respectively, so that $m_{ij}=1$ if the point $P_j$ and the block $x_i$ are incident, and $m_{ij}=0$ otherwise. An isomorphism from one incidence structure to another is a bijective mapping of points to points and from blocks to blocks that preserves the incidence. An isomorphism from an incidence structure $\mathcal{D}$ onto itself is called an automorphism of $\mathcal{D}$. The set of all automorphisms forms a group, the full automorphism group of $\mathcal{D}$, and is denoted by Aut$(\mathcal{D})$.

A $t-(v,k,\lambda)$ design is a finite incidence structure $\mathcal{D} = (\mathcal{P}, \mathcal{B}, \mathcal{I})$ that satisfies the following conditions: $\mathcal{P}$ has exactly $v$ elements, each element of $\mathcal{B}$ is incident with exactly $k$ elements of $\mathcal{P}$ and every $t$ elements of $\mathcal{P}$ are incident with exactly $\lambda$ elements of $\mathcal{B}$.

We assume that the reader is familiar with the basics of permutation group theory. For basic definitions and group theoretical terms, we refer the reader to \cite{dixon}. For the action of a group $G$ on a set $\Omega$  we will denote the image of $x\in \Omega$ for the action of $g\in G$ by $g.x$.

Let $G$ be a finite permutation group acting primitively on sets $\Omega_1$ and $\Omega_2$. In \cite{vedrana2} a construction of a 1-design with block set $\Omega_1$ and point set $\Omega_2$, with $G$ as automorphism group, was presented. This construction allows us to construct 1-designs that are not necessarily symmetric, and stabilizers of a point and a block are not necessarily conjugate. Using this construction, block designs and strongly regular graphs were constructed in \cite{vedrana1} and \cite{dean1}. The construction described in \cite{vedrana2} gives us all designs on which the group $G$ acts primitively on points and blocks.

A generalization of this construction is described in \cite{dean1} and is used to construct not necessarily primitive, but still transitive block designs.

\begin{thm}\label{tm_clanak}
	Let $G$ be a finite permutation group acting transitively on the sets $\Omega_1$ and $\Omega_2$ of size $m$ and $n$, respectively. Let $\alpha\in\Omega_1$ and $\Delta_2=\displaystyle\cup_{i=1}^{s}G_\alpha.\delta_i$, where $\delta_1,\dots,\delta_s\in\Omega_2$ are representatives of distinct $G_\alpha$-orbits. If $\Delta_2\ne \Omega_2$ and
	$\mathcal{B}=\{g.\Delta_2\  |\  g\in G\},$
	then $\mathcal{D}=(\Omega_2,\mathcal{B})$ is a $1$-$(n,|\Delta_2|,\frac{|G_\alpha|}{|G_{\Delta_2}|}\sum_{i=1}^{s}|G_{\delta_i}.\alpha|)$ design with $\frac{m\cdot |G_\alpha|}{|G_{\Delta_2}|}$ blocks. The group $H\cong G/\cap_{x\in\Omega_2}G_x$ acts as an automorphism group on $(\Omega_2,\mathcal{B})$, transitively on points and blocks of the design.
\end{thm}

\begin{rem}\label{napomena_paired}
	If a group $G$ acts transitively on $\Omega$, and $\Delta$ is an orbit of $G_\alpha$, then $\Delta^{'}=\{g.\alpha\ |\ g\in G,g^{-1}.\alpha\in\Delta\}$ is also an orbit of $G_\alpha$. $\Delta^{'}$ is the orbit of $G_\alpha$ paired with $\Delta$. It is obvious that $\Delta^{''}=\Delta$ and $|\Delta^{'}|=|\Delta|$. If $\Delta^{'}=\Delta$, then $\Delta$ is said to be self-paired. 
\end{rem}

Theorem \ref{tm_clanak} gives us the construction of transitive block designs. If $\Omega_1=\Omega_2$ and $\Delta_2$ is a union of self-paired and mutually-paired orbits of $G_\alpha$, then the design $\mathcal{D}$ is a symmetric self-dual design and the incidence matrix of this design is the adjacency matrix of a $|\Delta_2|$-regular graph.

\begin{cor}
If a group $G$ acts transitively on the points and the blocks of a 1-design $\mathcal{D}$, then $\mathcal{D}$ can be obtained as described in Theorem \ref{tm_clanak}, \textit{i.e.}, such that $\Delta_2$ is a union of $G_\alpha$-orbits.
\end{cor}

The construction used in this paper is a modification of the construction described in Theorem \ref{tm_clanak}. 

\subsection{Directed regular graphs}\label{subsec:1.2}

A directed graph or digraph $\mathcal{G}$ is an ordered pair $(\mathcal{V},\mathcal{A})$ of a non-empty finite set $\mathcal{V}(\mathcal{G})$, whose elements are called vertices, and a finite family $\mathcal{A}(\mathcal{G})$ of ordered pairs of elements of $\mathcal{V}(\mathcal{G})$ called arcs. For two distinct vertices $x$ and $y$ of $\mathcal{G}$, if $(x,y)\in\mathcal{A}$, we write $x\to y$, and say that $x$ is an \textit{in-neighbor} of $y$ or $y$ is an \textit{out-neighbor} of $x$. We say that vertices $x$ and $y$ are adjacent if there is an arc $x\to y$ or $y\to x$.  For two distinct vertices $x$ and $y$ of $\mathcal{G}$, if $(x,y)\in\mathcal{A}$ and $(y,x)\in\mathcal{A}$, we say that $x\leftrightarrow y$ is an (undirected) edge. The out-degree of vertex $x$ is defined as the number of arcs of the form $x\to y$, while the in-degree of vertex $y$ denotes the number of such arcs that are directed to $y$. A directed graph is $k$-regular if for each vertex the number of in- and out-neighbors is $k$. A path of length $l$ from $x$ to $y$ in $\mathcal{G}$ is a finite sequence of distinct vertices $(x=u_0,u_1,u_2,\dots,u_l=y)$ such that $(u_{i-1},u_i)\in\mathcal{A}$ for $i=1,2,\dots,l$. The digraphs in this paper will not have arcs from a vertex to itself.

We define a square $\{0, 1\}$-matrix $A = (a_{xy})$ labeled with the vertices of $\mathcal{G}$ such that $a_{xy} = 1$ if there is an arc from a vertex $x$ to a vertex $y$. The matrix $A$ is called the adjacency matrix of the digraph $\mathcal{G}$. A directed graph where for every arc $(x,y)\in\mathcal{A}$, $x\ne y$, there is an arc $(y,x)\in\mathcal{A}$ is an (undirected) graph.

\subsubsection{Directed strongly regular graphs}\label{subsubsec:1.2.1}

A strongly regular digraph with parameters $(n,k,t,\lambda,\mu)$, denoted DSRG$(n,k,t,\lambda,\mu)$, is a $k$-regular digraph on $n$ vertices such that each vertex is incident with $t$ undirected edges, for any two distinct vertices $x$, $y$ the number of paths of length 2 from $x$ to $y$ is $\lambda$ if $x\to y$ and $\mu$ otherwise.

Let $A$ be the adjacency matrix of a directed graph $\mathcal{G}$ with $n$ vertices. Then $\mathcal{G}$ is a directed strongly regular graph with parameters $(n,k,t,\lambda,\mu)$ if and only if $AJ_n=J_nA=kJ_n$ and $A^2=tI_n+\lambda A+\mu (J_n-I_n-A)$, where $I_n$ is the identity matrix of order $n$ and $J_n$ is the all-one matrix of order $n$.

It is easy to see that the directed strongly regular graphs with $t=k$ are strongly regular graphs. The complement of a directed strongly regular graph is again a directed strongly regular graph.
\begin{prop}(\cite{duval})
If $\mathcal{G}$ is a directed strongly regular graph with parameters $(n,k,t,\lambda,\mu)$, then the complementary graph $\bar{\mathcal{G}}$ is a directed strongly regular graph with parameters $(n,\bar{k},\bar{t},\bar{\lambda},\bar{\mu})$, where $\bar{k}=n-k-1$, $\bar{t}=n-2k+t-1$, $\bar{\lambda}=n-2k+\mu-2$, $\bar{\mu}=n-2k+\lambda$.
\end{prop}
By the eigenvalues of a directed strongly regular graph we mean the eigenvalues of its adjacency matrix. It is well known that a directed strongly regular graph has exactly three distinct eigenvalues. More precisely, a directed strongly regular graph with parameter set $(n,k,t,\lambda,\mu)$ has eigenvalues
$$
\theta_0=k,\qquad
\theta_{1,2}=\frac{1}{2}\left(\lambda-\mu\pm\sqrt{(\mu-\lambda)^2+4(t-\mu)}\right),
$$
with respective multiplicities
$$
m_0=1,\qquad
m_1=\frac{k+\theta_2(n-1)}{\theta_2-\theta_1},\qquad
m_2=\frac{k+\theta_1(n-1)}{\theta_1-\theta_2}.
$$

\subsubsection{Quasi-strongly regular digraphs}\label{subsubsec:1.2.2}

A quasi-strongly regular digraph $\mathcal{G}$ with parameters $(n,k,t,a; c_1,c_2,\dots, c_p)$, which we denote by \\
\noindent QSRD$(n,k,t,a; c_1,c_2,\dots, c_p)$, is a $k$-regular digraph on $n$ vertices such that each vertex is incident with $t$ undirected edges, for any two distinct vertices $x$, $y$ the number of paths of length 2 from $x$ to $y$ is $a$ if $x\to y$ and an element of $\{c_1,\ldots , c_p\}$ otherwise, and for each $i\in \{1,\ldots , p\}$ there are two distinct vertices $x\nrightarrow y$ such that the number of paths of length 2 from $x$ to $y$ is $c_i$. The parameter $p$ is called the grade of $\mathcal{G}$. Without loss of generality, we assume that $c_1 > c_2>\cdots > c_p$. Let $A$ be the adjacency matrix of a directed graph $\mathcal{G}$ with $n$ vertices. Then $\mathcal{G}$ is a quasi-strongly regular digraph with parameters $(n,k,t,a;c_1,\dots,c_p)$ if and only if $AJ_n=J_nA=kJ_n$ and $A^2=tI_n+aA+c_1C_1+c_2C_2+\cdots+c_pC_p$ for some non-zero $(0,1)$- matrices $C_1,C_2,\dots,C_p$ such that $C_1+C_2+\cdots+C_p=J_n-I_n-A$. 

A quasi-strongly regular digraph is a directed strongly regular graph if $p=1$ and it is a quasi-strongly regular graph if $t=k$.

\section{Construction}\label{sec:2}

In the following theorem we present a construction of vertex-transitive directed regular graphs.

\begin{thm}\label{glavni_tm}
	Let $G$ be a finite permutation group acting transitively on the set $\Omega$. Let $\alpha\in\Omega$ and let $\Delta=\displaystyle \cup_{i=1}^{s}G_\alpha.\delta_i$ be a union of orbits of the stabilizer $G_\alpha$ of $\alpha$, where $\delta_1,\dots,\delta_s$ are representatives of distinct $G_\alpha$-orbits. Let $T=\{g_1,\dots,g_t\}$ be a set of representatives of left cosets in $G/G_\alpha=\{g_1G_\alpha,\dots,g_tG_\alpha\}$. Let $\mathcal{V}=\{g_i.\alpha\ |\ i=1,\dots,t\}$ and let $\mathcal{A}=\{(g_i.\alpha,g_i.\beta)\ |\ i=1,\dots,t,\beta\in\Delta\}$. Then $\mathcal{G}=(\mathcal{V},\mathcal{A})$ is a directed graph on $|\Omega|$ vertices that is $|\Delta|$-regular and such that $g_i.\Delta$ is a set of out-neighbors of the vertex $g_i.\alpha$, $i=1,\dots,t$.
\end{thm}
\begin{proof} A group $G$ acts transitively on the set $\Omega$, so that for each $y\in\Omega$ there is an element of a transversal $T=\{g_1,\dots,g_t\}$ such that $y=g_i.\alpha$, $i=1,\dots,t$. Therefore $|\mathcal{V}|=|\Omega|$.
		
		A set $\Delta$ is the union of orbits of the stabilizer $G_\alpha$ for the action of $G$.  For $\beta\in\Delta$, $g.(\alpha,\beta)=(\alpha,g.\beta)\Leftrightarrow g\in G_\alpha$, so $\Delta$ is the set of all out-neighbors of a vertex $\alpha$.
		
		Let $x\in\mathcal{V}$, $x\ne \alpha$. Then there is a unique transversal element $g_i$ such that $x=g_i.\alpha$ and
		for each $\beta\in\Delta$ there is an arc $x\to y$, where $y=g_i.\beta$.
		So $g_i.\Delta$ is a set of out-neighbors of a vertex $x\ne \alpha$ and out-degree of $x$ is $|g_i.\Delta|=|\Delta|$.
		
		Assume that for $y\in \Omega$ there exist $h_1,\dots,h_s\in T$ such that $y=h_i.\beta$ for some $\beta\in\Delta$. Then the in-degree of vertex $y$ is equal to $s$, and the vertices $h_i.\alpha$, $i\in \{1,\ldots,s\}$, are in-neighbors of $y$. Due to transitivity, for every other vertex $y^{'}\ne y$ there exists $g^{'}\in G$ such that $y^{'}=g^{'}.y$, \textit{i.e.}, $g^{'}h_i.\alpha$, $i\in \{1,\ldots,s\}$, are in-neighbors of $y^{'}$. It follows that each  vertex has the in-degree $s$.

		If we count the 1s in the rows and columns of the adjacency matrix of the digraph, we get
		$|\Omega|\cdot |\Delta|=|\Omega|\cdot s\Rightarrow s=|\Delta|. $
\end{proof}
		A group $G$ acts transitively as an automorphism group on a set of vertices of a digraph and non-transitively on a set of arcs.

\begin{rem}
	\begin{enumerate}
		\item If $\Delta$ is a single orbit, the constructed digraph is arc-transitive (for more information see \cite{arctr}). 
		\item If $\Delta$ is a union of mutually paired suborbits and self-paired suborbits, then $\mathcal{G}$ is an undirected graph. Otherwise, $\mathcal{G}$ is a directed graph. For more information, we refer the reader to \cite{dixon}. 
		\item Since for $g\in G_\alpha$, $g.\Delta=\Delta$, it follows that $G_\alpha\subseteq G_\Delta$, where $G_\Delta$ is a setwise stabilizer of $\Delta$. Since $\Delta$ is a union of $G_\alpha$-orbits, $G_\Delta$ is a group and $G_\alpha\leq G_\Delta$. 
		\begin{itemize}
			\item If $G_\alpha=G_\Delta$, then the sets $g_i.\Delta$, $g_i\in T$, of out-neighbors of vertices $g_i.\alpha$, $i=1,\dots,t$, are all distinct.
			\item If $G_\alpha<G_\Delta$, then there exist vertices $g_i.\alpha$ and $g_j.\alpha$, $g_i, g_j \in T$, such that the sets $g_i.\Delta$ and $g_j.\Delta$ are equal.
		\end{itemize}
		\item For the action of a group $G$, a $G_\alpha$-orbit $\{\alpha\}$ is a self-paired orbit and is called diagonal. If we include the diagonal in the union $\Delta$, we obtain a digraph containing all the loops.
	\end{enumerate}
\end{rem}

Every directed regular graph for which a group $G$ acts transitively on a set of vertices of this digraph is isomorphic to a digraph constructed using Theorem \ref{glavni_tm}.

\begin{thm}
	If a group $G$ acts transitively as an automorphism group on a set of vertices of a directed regular graph $\mathcal{G}=(\mathcal{V},\mathcal{A})$, then there is such a set $\Omega$ that vertices and arcs of a digraph $\mathcal{G}$ are defined in the way described in Theorem \ref{glavni_tm}.
\end{thm}
\begin{proof}
Let $\alpha$ be a vertex of a digraph $\mathcal{G}$ and let $\Delta$ be a set of out-neighbors of $\alpha$.

Let $\delta_1\in\Delta$. Consider the action of a group $G$ on a vertex $\delta_1$. The orbit $G_\alpha.\delta_1$ is contained in the set $\Delta$, because for $g\in G_\alpha$ it holds that $g.(\alpha,\delta_1)=(\alpha,g.\delta_1)$. If $G_\alpha.\delta_1\ne \Delta$, then $G_\alpha.\delta_2\subseteq \Delta$, for some vertex $\delta_2\in\Delta$ that is not contained in the orbit $G_\alpha.\delta_1$. Therefore, a set $\Delta$ is the union of $G_\alpha$- orbits, \textit{i.e.} $\Delta=\displaystyle\cup_{i=1}^{s}G_\alpha.\delta_i$, for some $s\in\mathbf{N}$.

A group $G$ acts transitively on the set of vertices of a digraph $\mathcal{G}$ so $\mathcal{V}=G.\alpha$.
\end{proof}

\subsection{The single-orbit case}\label{subsec:2.1}

Although Theorem~\ref{glavni_tm} allows arbitrary unions of suborbits, the case in which
$\Delta$ is a single suborbit is a direct instance of the association-scheme construction of
quasi-strongly regular digraphs. Indeed, the orbitals of a transitive permutation group form
a Schurian association scheme (see \cite{as}), and Theorem~3.10 of \cite{quasi} applies to each
non-diagonal orbital. We record the corresponding statement in the notation used here.

\begin{cor}\label{cor_single_orbit_qsrd}
Let $G$ be a finite permutation group acting transitively on a set $\Omega$, let
$\alpha\in\Omega$, and let $\Delta=G_\alpha.\delta$ be a single non-diagonal
$G_\alpha$-orbit. Let $\mathcal{G}=(\Omega,\mathcal{A})$ be the digraph obtained by
Theorem~\ref{glavni_tm}. Let
$\Delta_0=\{\alpha\},\Delta_1=\Delta,\Delta_2,\dots,\Delta_{r-1}$ be all distinct
$G_\alpha$-orbits on $\Omega$, and choose representatives $\delta_j\in\Delta_j$,
$j=0,1,\dots,r-1$. Let $\Delta'$ be the orbit paired with $\Delta$, and put
$$
t=|\Delta\cap\Delta'|.
$$
For $j=1,\dots,r-1$, define
$$
m_j=\bigl|\{z\in\Omega\mid \alpha\to z\to \delta_j\}\bigr|.
$$
If $c_1>\cdots>c_p$ are the distinct values among $m_2,\dots,m_{r-1}$, then
$\mathcal{G}$ is a directed quasi-strongly regular graph with parameters
$(|\Omega|,|\Delta|,t,m_1;c_1,\dots,c_p)$.
\end{cor}

\begin{proof}
For $j=0,1,\dots,r-1$, let
$$
R_j=\{(g.\alpha,g.\delta_j)\mid g\in G\}.
$$
The relations $R_0,R_1,\dots,R_{r-1}$ are the orbitals of the action of $G$ on
$\Omega\times\Omega$ and form the associated Schurian association scheme. Since
$\Delta=\Delta_1$, the arc set of $\mathcal{G}$ is precisely $R_1$. The assertion now
follows from Theorem~3.10 of \cite{quasi}. In this notation, the corresponding
intersection numbers are
$$
p^0_{11}=|\Delta\cap\Delta'|=t,\qquad p^1_{11}=m_1,\qquad p^j_{11}=m_j
\quad \text{for } j=2,\dots,r-1,
$$
which gives the stated parameters.
\end{proof}

In particular, the grade-one case gives directed strongly regular graphs. Namely, if $m_2=m_3=\cdots=m_{r-1}=\mu$, then $\mathcal{G}$ is a directed strongly regular graph with parameters
$(|\Omega|,|\Delta|,t,m_1,\mu)$.

\subsection{The complement of a directed quasi-strongly regular graph}\label{subsec:2.2}

The complement of a directed quasi-strongly regular graph need not be quasi-strongly regular. Nevertheless, in the present construction the complement is obtained by the same group action, simply by replacing the chosen set of suborbits with its complement in $\Omega\setminus\{\alpha\}$.

\begin{rem}
 Let a group $G$ act transitively on a set $\Omega$ and let $\Delta_0=\{\alpha\}$, $\Delta_1,\dots,\Delta_{r-1}$ be $G_\alpha$-orbits for that action. Let $\mathcal{G}$ be a directed quasi-strongly regular graph constructed by the method described in Theorem \ref{glavni_tm}. If a set $\displaystyle\cup_{i\in S}\Delta_i$, $S\subseteq\{1,\dots,r-1\}$, is the set of out-neighbors of a vertex $\alpha\in\Omega$ of the directed quasi-strongly regular graph, then a union $\Delta=\displaystyle\cup_{j=1}^{r-1} \Delta_j\backslash \displaystyle\cup_{i\in S}\Delta_i$ is a set of out-neighbors of a vertex $\alpha$ of the complement of the directed quasi-strongly regular graph.
\end{rem}

The complement of a directed quasi-strongly regular graph is not necessarily a directed quasi-strongly regular graph. As a consequence of Theorem \ref{tm_parameters}, which we prove here, we conclude the following.

The complement of a QSRD$(n,k,t,a;c_1,\dots,c_p)$ is a directed $k'$-regular graph on $n$ vertices, for which the following holds:
\begin{itemize}
	\item[(i)] each vertex is incident to $t'$ undirected edges,
	\item[(ii)] there exists a number $a'$ such that for every two distinct vertices $x$ and $y$, such that $x\not\to y$, the number of paths of length two from $x$ to $y$ is $a'$,
	\item[(iii)] there exist such numbers $b_i$ that for every two vertices $x$ and $y$, such that $x\to y$, the number of paths of length two from $x$ to $y$ is $b_i$, where $1\leq i \leq p$,
	\item[(iv)] for each $1\leq i \leq p$ there are distinct vertices $x$ and $y$, such that $x\to y$, so that the number of paths of length two from $x$ to $y$ is equal to $b_i$.
\end{itemize}
We will denote a directed graph with these properties by QSRD$^C(n,k',t',a';b_1,\dots,b_p)$. We call a non-negative integer $p$ the grade of this digraph. Without loss of generality, we assume that $k'>0$ and $b_1>b_2>\cdots > b_p$. 

For $t'=k'$ we obtain an undirected graph.

\begin{thm}\label{tm_parameters}
	If $\mathcal{G}$ is a QSRD$(n,k,t,a;c_1,\dots,c_p)$ with adjacency matrix $A$, then the complementary graph $\mathcal{G}'$ with adjacency matrix $A'=J-I-A$ is a directed graph QSRD$^C(n,k',t',a';b_1,\dots,b_p)$ and
	$$A'J=JA'=(n-k-1)J,$$ $$ (A')^2=t'I+a'(J-I-A')+b_1C_1+\cdots+b_pC_p,$$
	where $C_1+\cdots+C_p=J-I-A=A'$ and the parameters are
	$$k'=(n-2k)+(k-1),\: t'=(n-2k)+(t-1),\: a'=(n-2k)+a,$$
	$$b_1=(n-2k)+(c_1-2),\: b_2=(n-2k)+(c_2-2),\dots,b_p=(n-2k)+(c_p-2).$$
\end{thm}
\begin{proof}
Adjacency matrix $A$ of a directed quasi-strongly regular graph with parameters $(n,k,t,a;c_1,\dots,c_p)$ satisfies:
$$	AJ=JA=kJ,\label{f1}$$	$$A^2=tI+aA+c_1C_1+\cdots+c_pC_p, \label{f2}$$
where $C_1+\cdots+C_p=J-I-A.$
Adjacency matrix of a complementary graph of that digraph is $A'=J-I-A$. 

It follows:
$$A'J=(J-I-A)J=J^2-I\cdot J-A\cdot J=nJ-J-kJ=(n-k-1)J,$$
$$JA'=J(J-I-A)=J^2-J\cdot I-J\cdot A=nJ-J-kJ=(n-k-1)J,$$

\noindent and 

$$(A')^2=(J-I-A)^2$$ $$		=(n-2k+t-1)I+(n-2k+a)A+(n-2k-2)(J-I-A)+c_1C_1+\cdots+c_pC_p$$
	$$	=(n-2k+t-1)I+(n-2k+a)(J-I-A')+(n-2k-2+c_1)C_1+\cdots+
	 (n-2k-2+c_p)C_p,$$
where $C_1+\cdots+C_p=J-I-A=A'$. 
\end{proof}

\section{New directed strongly regular graphs}\label{sec:3}

In this section we give examples of the construction method and prove the existence of directed strongly regular graphs with parameters $(72,25,15,6,10)$, $(96,11,4,3,1)$, $(96,22,16,8,4)$,\\
\noindent $(96,26,11,7,7)$, $(96,29,23,10,8)$, $(96,42,32,16,20)$, $(96,45,35,22,20)$ and $(165,60,36,23,21)$. 

For the explicit search for new directed strongly regular graphs, it is natural to look for permutation groups whose actions are rich enough to produce several essentially different suborbit unions, but still sufficiently structured to make the orbit calculations and the verification of the directed strongly regular conditions computationally feasible. 

In all computations we used GAP \cite{GAP2022} together with the package Digraphs \cite{digraphs}. In our search we considered transitive permutation representations of the groups $GO^{\pm}(4,3)$ and $SO^{\pm}(4,3)$. The search was restricted to representations of rank at most 30 and, except for $SO^{+}(4,3)$, to permutation representations of degree $v\leq 110$. The group $SO^+(4,3)$ is a special case in this search: among its transitive permutation representations of rank at most 30, the largest degree that occurs is 96. These restrictions were imposed both to keep the computational search feasible and to focus on the range of orders covered by Brouwer's tables of feasible parameters for directed strongly regular graphs. In addition, we performed a separate search using the Mathieu group $M_{11}$. For this group we considered transitive permutation representations of rank at most 30, without imposing a restriction on the degree $v$.

For each group \(G\) in Table~1, we consider a transitive permutation representation
\(P\) of \(G\) on a set \(\Omega\) of size \(n\), where \(n\) is the number of vertices of the
constructed digraph. Fixing a point \(\alpha \in \Omega\), we compute the suborbits of
the point stabilizer \(G_\alpha\), and the directed strongly regular graphs listed below
are obtained, as in Theorem~3.1, from suitable unions of non-diagonal $G_\alpha$-orbits. In total, we obtain $26$ pairwise non-isomorphic directed strongly regular graphs, namely $\mathcal{G}_1,\dots,\mathcal{G}_{26}$, with eight distinct parameter sets.

\begin{table}[h!]\scriptsize
\centering
\caption{New directed strongly regular graphs arising from orthogonal and Mathieu group actions}\label{tab:1}
\begin{tabular}{cccccc}
\hline\noalign{\smallskip}
$G$& $P$ & $G_\alpha$& Parameters & $\#$-nonisom. & Aut$\mathcal{G}$ \\
\noalign{\smallskip}\hline\noalign{\smallskip}
$GO^{+}(4,3)$&$({A}_{4}^{2}:2):2)$&$D_4$&$(72,25,15,6,10)$& 2& $(((2^2\times(2^4:3)):2):3):2$\\
\cline{2-6}
&$(((2^4:3):2):3):2$&$S_3$&$(96,11,4,3,1)$& 4& $({A}_{4}^{2}:2):2$ (2), \\
&&&& & $((((({A}_{4}^{2}:2)\times A_4):2):3):2):2$ (2)\\
\cline{2-6}
&$({A}_{4}^{2}:2):2)$&$C_6$&$(96,22,16,8,4)$ & 2& $(((2^3:2^2):3^2):2):2$\\
\cline{2-6}
&$(((2^3:2^2):3^2):2):2$&$D_6$&$(96,22,16,8,4)$& 2 & $2\times((((2^4:3):2):3):2)$ \\
\cline{2-6}
&$({A}_{4}^{2}:2):2)$&$S_3$&$(96,26,11,7,7)$ & 8 & $({A}_{4}^{2}:2):2$\\
\cline{2-6}
&$(((2^4:3):2):3):2$&$S_3$&$(96,42,32,16,20)$& 2 & $({A}_{4}^{2}:2):2$\\
\cline{2-6}
&$(((2^4:3):2):3):2$&$S_3$&$(96,45,35,22,20)$&2& $({A}_{4}^{2}:2):2$\\
\hline
$SO^{+}(4,3)$&$((2^3:2^2):3^2):2$&$S_3$&$(96,29,23,10,8)$ & 2 & $((2^3:2^2):3^2):2$\\
\hline
$M_{11}$&$M_{11}$&$GL(2,3)$&$(165,60,36,23,21)$ & 2& $M_{11}$\\
\noalign{\smallskip}\hline
\end{tabular}
\end{table}

For each directed strongly regular graph listed in Table~1, its complementary digraph is again a directed strongly regular graph; moreover, the complements are pairwise non-isomorphic whenever the corresponding original digraphs are, and the full automorphism group of each complement is isomorphic to that of the corresponding constructed digraph.

All transitive permutation representations used in the search, together with the corresponding suborbits and the adjacency matrices of all newly constructed directed strongly regular graphs, are available at the following link:
\begin{center}
\url{https://www.math.uniri.hr/~matea.zubovic/dsrg/new_DSRGs.html}
\end{center}

\subsection{Directed strongly regular graphs from groups $GO^{\pm}(4,3)$}\label{subsec:3.1}

The general orthogonal groups $GO^{+}(4,3)$ and $GO^{-}(4,3)$ over $\mathbb{F}_3$ are small classical groups, and the distinction between the plus and minus type leads to related but genuinely different orbit structures. Our search in the corresponding action of $GO^{\pm}(4,3)$ also produced several directed strongly regular graphs with previously known parameter sets, so we do not describe them in detail here. Digraphs with parameters $(72,25,15,6,10)$, $(96,11,4,3,1)$, $(96,22,16,8,4)$, $(96,26,11,7,7)$, $(96,42,32,16,20)$ and $(96,45,35,22,20)$ are obtained from the group $GO^{+}(4,3)$. To the best of our knowledge, these directed strongly regular graphs have not appeared previously in the literature. No new directed strongly regular graphs were obtained from the group $GO^{-}(4,3)$.

	\subsubsection{DSRG(72,25,15,6,10)}\label{subsubsec:3.1.2}

For the degree $72$ representation of $GO^{+}(4,3)$, applying
Theorem~\ref{glavni_tm} to the unions
$\Delta_{\mathcal{G}_1}=
\bigcup_{i\in I}\Delta_i$, $I=\{1,4,9,10,14,15,17,19,21\}$ and $\Delta_{\mathcal{G}_2}=
\bigcup_{j\in J}\Delta_j$, $J=\{1,4,9,10,14,15,17,19,20\}$
produces two non-isomorphic directed strongly regular graphs with parameters
$(72,25,15,6,10)$.

\subsubsection{DSRG(96,11,4,3,1)}\label{subsubsec:3.1.3}

For the degree $96$ representation of $GO^{+}(4,3)$ of rank $30$, the unions
$\Delta_{\mathcal{G}_3}=\displaystyle\bigcup_{i\in I}\Delta_i$, $I=\{1,2,7,16,21\}$, $\Delta_{\mathcal{G}_4}=\displaystyle\bigcup_{j\in J}\Delta_j$, $J=\{1,2,9,11,21\}$, $\Delta_{\mathcal{G}_5}=\displaystyle\bigcup_{k\in K}\Delta_k$, $K=\{1,2,11,14,23\}$ and $\Delta_{\mathcal{G}_6}=\displaystyle\bigcup_{l\in L}\Delta_l$, $L=\{1,2,14,16,23\}$ give four pairwise non-isomorphic directed strongly regular graphs with parameters $(96,11,4,3,1)$.

\subsubsection{DSRG(96,22,16,8,4)}\label{subsubsec:3.1.4}

The parameter set $(96,22,16,8,4)$ is obtained from two degree $96$ representations of
$GO^{+}(4,3)$. For the first representation, the unions $\Delta_{\mathcal{G}_7}=\displaystyle\bigcup_{i\in I}\Delta_i$, $I=\{2,3,6,15,16\}$ and $\Delta_{\mathcal{G}_8}=\displaystyle\bigcup_{j\in J}\Delta_j$, $J=\{2,3,6,15,17\}$
give two directed strongly regular graphs. For the second
representation, the unions
$\Delta_{\mathcal{G}_{9}}=\displaystyle\bigcup_{k\in K}\Delta_k$, $K=\{2,3,7,10,11\}$ and $\Delta_{\mathcal{G}_{10}}=\displaystyle\bigcup_{l\in L}\Delta_l$, $L=\{2,3,7,10,13\}$
give two further examples. Hence, we obtain four pairwise non-isomorphic directed
strongly regular graphs with parameters $(96,22,16,8,4)$.	

\subsubsection{DSRG(96,26,11,7,7)}\label{subsubsec:3.1.5}

In another degree $96$ representation of $GO^{+}(4,3)$, the following eight unions of suborbits give directed strongly regular graphs:\\
\begin{minipage}{0.48\textwidth}
$\Delta_{\mathcal{G}_{11}}=\displaystyle\bigcup_{i\in I}\Delta_i$, $I=\{1,2,3,4,7,9,17,21\}$,

$\Delta_{\mathcal{G}_{12}}=\displaystyle\bigcup_{j\in J}\Delta_j$, $J=\{1,2,3,4,7,9,16,19\}$,

$\Delta_{\mathcal{G}_{13}}=\displaystyle\bigcup_{k\in K}\Delta_k$, $K=\{1,2,3,4,7,9,13,17\}$,

$\Delta_{\mathcal{G}_{14}}=\displaystyle\bigcup_{l\in L}\Delta_l$, $L=\{1,2,3,4,7,9,13,16\}$,
\end{minipage}
\hfill
\begin{minipage}{0.48\textwidth}
$\Delta_{\mathcal{G}_{15}}=\displaystyle\bigcup_{m\in M}\Delta_m$, $M=\{1,2,3,4,7,8,17,21\}$,

$\Delta_{\mathcal{G}_{16}}=\displaystyle\bigcup_{r\in R}\Delta_r$, $R=\{1,2,3,4,7,8,16,19\}$,

$\Delta_{\mathcal{G}_{17}}=\displaystyle\bigcup_{s\in S}\Delta_s$, $S=\{1,2,3,4,7,8,14,17\}$, 

$\Delta_{\mathcal{G}_{18}}=\displaystyle\bigcup_{t\in T}\Delta_t$, $T=\{1,2,3,4,7,8,14,16\}$.
\end{minipage}\\
\noindent These unions produce eight pairwise non-isomorphic directed strongly regular graphs with parameters $(96,26,11,7,7)$.

\subsubsection{DSRG(96,42,32,16,20)}\label{subsubsec:3.1.6}

For the degree $96$ representation of $GO^{+}(4,3)$ of rank $30$, the unions $\Delta_{\mathcal{G}_{19}}=\displaystyle\bigcup_{i\in I}\Delta_i$, $I=\{1,\allowbreak2,\allowbreak3,\allowbreak6,\allowbreak10,\allowbreak13,\allowbreak14,\allowbreak17,\allowbreak18,\allowbreak20,\allowbreak24,\allowbreak26,\allowbreak28\}$ and $\Delta_{\mathcal{G}_{20}}=\displaystyle\bigcup_{j\in J}\Delta_j$, $J=\{1,\allowbreak2,\allowbreak3,\allowbreak9,\allowbreak10,\allowbreak13,\allowbreak17,\allowbreak18,\allowbreak19,\allowbreak20,\allowbreak24,\allowbreak26,\allowbreak28\}$, give two non-isomorphic directed strongly regular graphs
with parameters $(96,42,32,16,20)$.

\subsubsection{DSRG(96,45,35,22,20)}\label{subsubsec:3.1.7}

For the degree $96$ representation of $GO^{+}(4,3)$ of rank $30$, the unions $\Delta_{\mathcal{G}_{21}}=\displaystyle\bigcup_{i\in I}\Delta_i$, $I=\{1,\allowbreak2,\allowbreak3,\allowbreak6,\allowbreak10,\allowbreak11,\allowbreak13,\allowbreak14,\allowbreak16,\allowbreak17,\allowbreak18,\allowbreak20,\allowbreak22,\allowbreak24,\allowbreak26\}$ and $\Delta_{\mathcal{G}_{22}}=\displaystyle\bigcup_{j\in J}\Delta_j$, $J=\{1,\allowbreak2,\allowbreak3,\allowbreak9,\allowbreak10,\allowbreak11,\allowbreak13,\allowbreak16,\allowbreak17,\allowbreak18,\allowbreak19,\allowbreak20,\allowbreak22,\allowbreak24,\allowbreak26\}$, give two non-isomorphic directed strongly regular graphs
with parameters $(96,45,35,22,20)$.

\subsection{Directed strongly regular graphs from group $SO^{+}(4,3)$}\label{subsec:3.2}

In addition to the groups $GO^{\pm}(4,3)$, we also considered the corresponding special orthogonal groups $SO^{\pm}(4,3)$, since they naturally arise as index-two subgroups of $GO^{\pm}(4,3)$. Although these groups are slightly smaller, their permutation actions still exhibit a rich suborbit structure. In particular, the actions of $SO^{\pm}(4,3)$ provide transitive permutation representations whose point stabilizers produce several non-trivial suborbits, making them suitable candidates for constructing directed strongly regular graphs via unions of suborbits.

The group $SO^{-}(4,3)$ with rank at most 30 and $v\leq 110$ did not yield any new directed strongly regular graphs. However, the group $SO^{+}(4,3)$ produced directed strongly regular graphs with parameters $(96,29,23,10,8)$. To the best of our knowledge, these directed strongly regular graphs are new.

\subsubsection{DSRG(96,29,23,10,8)}\label{subsubsec:3.2.1}

For the degree $96$ representation of $SO^{+}(4,3)$, the unions $\Delta_{\mathcal{G}_{23}}=\displaystyle\bigcup_{i\in I}\Delta_i$, $I=\{1,7,8,9,14,17,18\}$ and $\Delta_{\mathcal{G}_{24}}=\displaystyle\bigcup_{j\in J}\Delta_j$, $J=\{1,7,8,9,12,17,18\}$ give two non-isomorphic directed strongly regular graphs
with parameters $(96,29,23,10,8)$.

\subsection{Directed strongly regular graphs from Mathieu group $M_{11}$}\label{subsec:3.3}

The Mathieu group $M_{11}$ is a sporadic simple group. The suborbits of a point stabilizer may interact in non-obvious ways, and suitable unions of them can lead to parameter sets that are not suggested by more routine constructions. For general information about Mathieu groups, we refer the reader to \cite{atlas}.  

\subsubsection{\textbf{DSRG(165,60,36,23,21)}}\label{subsubsec:3.3.1}

For the primitive degree $165$ representation of the Mathieu group $M_{11}$, the unions $\Delta_{\mathcal{G}_{25}}=\Delta_2\cup \Delta_3\cup \Delta_5$ and 
$\Delta_{\mathcal{G}_{26}}=\Delta_2\cup \Delta_3\cup \Delta_6$ give two non-isomorphic directed strongly regular graphs
with parameters $(165,60,36,23,21)$. This parameter set does not belong to any known
family of directed strongly regular graphs listed in \cite{dsrgbaza}.

\section{Classification of vertex-transitive and vertex-primitive quasi-strongly regular digraphs}\label{sec:4}

In this section we give information about directed quasi-strongly regular graphs constructed under the action of transitive permutation groups of degree $n$, $n\in\{3,\dots,30\}$, and primitive permutation groups of degree $n$, $n\in\{31,\dots,110\}$, up to isomorphism. For primitive groups, the classification was performed for permutation representations of rank at most $30$. 

Since directed strongly regular graphs form an important special case of directed quasi-strongly regular graphs, it is natural to compare our classification with earlier results on vertex-transitive directed strongly regular graphs. Several authors have studied directed strongly regular graphs with transitive automorphism groups. Fiedler, Klin and Muzychuk in \cite{FKM2002} classified small vertex-transitive directed strongly regular graphs, while the authors in \cite{rez2} constructed infinite families of vertex-transitive directed strongly regular graphs.  Gy\"urki later in \cite{Gyurki2021} extended the enumeration of vertex-transitive directed strongly regular graphs to order $31$. 

The classification of vertex transitive directed strongly regular graphs can be efficiently approached via coherent configurations, or equivalently non-commutative association schemes, and suitable mergings of their relations (\cite{FKM2002,KMMZ2004}). In this paper, we applied the construction from the Theorem \ref{glavni_tm} for transitive permutation groups of degree $n\in\{3,\dots,30\}$ and obtained a complete classification of quasi-strongly regular digraphs. The suborbit-union construction from Theorem \ref{glavni_tm} provides a uniform framework for constructing and classifying quasi-strongly regular digraphs from transitive and primitive permutation groups, with directed strongly regular graphs occurring as the special case of grade one. The resulting classification provides a database of vertex-transitive quasi-strongly regular digraphs. A further motivation for this exhaustive search is that the adjacency matrices of quasi-strongly regular digraphs obtained in this way provide a rich and systematic source of input matrices for coding-theoretic constructions. In particular, they can be used to construct linear codes, including self-orthogonal codes, thereby giving an additional application of quasi-strongly regular digraphs beyond graph classification.
	
The results are described by the following theorem. 

	\begin{thm}
		Up to isomorphism, there are $45587$ directed quasi-strongly regular graphs, admitting a vertex-transitive automorphism group of degree $n$, $n\in\{3,\dots,30\}$. $472$ of them are directed strongly regular graphs with $0<t<k$. 
	\end{thm}

Also, for primitive permutation groups of degree $n\in\{31,\dots,110\}$ and rank at most $30$, using the construction method from Theorem \ref{glavni_tm} we obtain a complete classification of quasi-strongly regular digraphs described by the following theorem.
	
	\begin{thm}
	There are, up to isomorphism, $5517$ directed quasi-strongly regular graphs
	constructed under the action of vertex-primitive permutation groups of degree
	$n\in\{31,\dots,110\}$ and rank at most $30$.
	\end{thm}

The computation showed that no vertex-primitive directed strongly regular graph with $0<t<k$ occurs in the considered primitive permutation representations.
 
	\begin{thm}
		There is no directed strongly regular graph on $n$ vertices, $n\in\{31,\dots,110\}$, with $0<t<k$, such that the primitive automorphism group $G$, of rank at most $30$, acts on the set of vertices.
	\end{thm}
	
The parameters of all directed quasi-strongly regular graphs constructed from transitive permutation groups of degrees $3$ to $30$, and from primitive permutation groups of degrees $31$ to $110$, up to isomorphism, together with their adjacency matrices and information about their full automorphism groups, are available at
\begin{center}
\url{https://www.math.uniri.hr/~matea.zubovic/qsrd/QSRDs.html}.
\end{center}

\begin{rem}
Gy\"urki in \cite{Gyurki2021} extended the enumeration of vertex-transitive directed strongly regular graphs to order 31, obtaining 140 equivalence classes and proving that the remaining open parameter sets of order at most 31 admit no vertex-transitive realizations. Consistently with this result, our construction does not produce any of these open parameter sets. Hence, if directed strongly regular graphs with such parameters exist, none of them can admit a vertex-transitive automorphism group.
\end{rem}

\section{Statements and Declarations}
\begin{itemize}
	\item Funding: This work has been fully supported by Croatian Science Foundation under the project 4571 and by the University of Rijeka under the project uniri-iskusni-prirod-23-155.
	\item Conflict of interest: The authors declare that they have no conflict of interest.
	\item Author Contributions: All authors contributed to the research and preparation of the manuscript. All authors read and approved the final manuscript.
	\item Data Availability Statement: The adjacency matrices of all constructed structures are publicly available and the download links are given in the article.
	All additional information about the structures is available on request.
\end{itemize}

\end{document}